\documentclass[a4paper,11pt]{article}
\usepackage{cmap}
\usepackage[T1]{fontenc}
\usepackage[utf8]{inputenc}
\usepackage[english]{babel}
\usepackage{amsfonts,amssymb,amsthm,amsmath}
\usepackage{url}
\usepackage{enumitem}

\usepackage{secdot}

\usepackage{graphicx}
\usepackage{epstopdf}
\usepackage{varwidth}

\usepackage{a4wide}
\usepackage{xcolor}
\usepackage[colorlinks=true,linktocpage,pdfpagelabels,
bookmarksnumbered,bookmarksopen]{hyperref}
\definecolor{ForestGreen}{rgb}{0.1,0.6,0.05}
\definecolor{EgyptBlue}{rgb}{0.063,0.1,0.6}
\hypersetup{
	colorlinks=true,
	linkcolor=EgyptBlue,         
	citecolor=ForestGreen,
	urlcolor=olive
}

\usepackage[hyperpageref]{backref}

\let\OLDthebibliography\thebibliography
\renewcommand\thebibliography[1]{
	\OLDthebibliography{#1}
	\setlength{\parskip}{1pt}
	\setlength{\itemsep}{1pt plus 0.3ex}
}

\numberwithin{equation}{section}
\newtheorem{theorem}{Theorem}[section]
\newtheorem{lemma}[theorem]{Lemma}
\newtheorem{prop}[theorem]{Proposition}

\theoremstyle{definition}
\newtheorem{remark}[theorem]{Remark}

\title{
	\vspace*{-1cm}
	Asymptotic relation for zeros of cross-product\\ of Bessel functions and applications
}

\makeatletter
\newcommand{\address}[1]{\gdef\@address{#1}}
\newcommand{\email}[1]{\gdef\@email{\url{#1}}}
\newcommand{\homepage}[1]{\gdef\@homepage{\url{#1}}}
\newcommand{\@endstuff}{\vspace{\baselineskip}\noindent\small
	\begin{tabular}{@{}l}
		\@address\\
		\textit{E-mail address:} \@email\\
		\textit{Homepage:} \@homepage
	\end{tabular}}
\AtEndDocument{\@endstuff}
\makeatother

\author{Vladimir Bobkov}
\email{bobkov@kma.zcu.cz}
\homepage{http://vladimir-bobkov.ru}
\address{
	\textsc{Department of Mathematics and NTIS, Faculty of Applied Sciences},\\ 
	\textsc{University of West Bohemia, Univerzitn\'i 8, 306 14 Plze\v{n}, Czech Republic};\\ 
	\textsc{Institute of Mathematics, Ufa Federal Research Centre, RAS},\\ 
	\textsc{Chernyshevsky str. 112, 450008 Ufa, Russia}}

\date{}

\begin{document}
\maketitle 

\begin{abstract}
	Let $a_{\nu,k}$ be the $k$-th positive zero of the cross-product of Bessel functions $J_\nu(R z) Y_\nu(z) - J_\nu(z) Y_\nu(R z)$, where $\nu\geq 0$ and $R>1$. 
	We derive an initial value problem for a first order differential equation whose solution $\alpha(x)$ characterizes the limit behavior of $a_{\nu,k}$ in the following sense:
	$$
	\lim_{k \to \infty} \frac{a_{kx,k}}{k} = \alpha(x),
	\quad 
	x \geq 0.
	$$
	Moreover, we show that 
	$$
	a_{\nu,k} < \frac{\pi k}{R-1} + \frac{\pi \nu}{2 R}.
	$$
	We use $\alpha(x)$ to obtain an explicit expression of the Pleijel constant for planar annuli and compute some of its values.
	
	
\end{abstract}

\section{Introduction and main results}

Consider the cross-product of the $\nu$-th order Bessel functions of the first and second kind
\begin{equation*}\label{eq:crossprod}
f_{\nu,R}(z) := J_\nu(R z) Y_\nu(z) - J_\nu(z) Y_\nu(R z).
\end{equation*}
Hereinafter, we always assume that $\nu, z \geq 0$, and $R>1$. 
However, since $f_{\nu,R}(z)$ is even with respect to $\nu$ and $z$ (see Appendix \ref{sec:appendix}), and $f_{\nu,R}(z) = -f_{\nu, 1/R}(Rz)$, the cases $\nu < 0$, $z < 0$, and $R \in (0,1)$ are also covered.
It is well-known that $f_{\nu,R}$ is oscillating and has infinitely many zeros all of which are simple, see the discussion in \cite{cochran}. 
We will denote by $a_{\nu,k}$ the $k$-th positive zero of $f_{\nu,R}$, $k = 1,2,\dots$

The prominent role of zeros of $f_{\nu,R}$ for applications (see, e.g., \cite{cochran3,kline}) reveals through the fact that $a_{\nu,k}^2$'s constitute the spectrum of the Laplace operator under homogeneous Dirichlet boundary conditions in a planar annulus with the inner radius $1$ and outer radius $R$. Below, we will use this relation and the main result of the present article to obtain a Pleijel-type result on the nodal domains statistics for the Laplace eigenfunctions in annuli \cite{BGS,pleijel}.

However, unlike the widely developed theory of zeros of Bessel functions $J_\nu$, $Y_\nu$, and corresponding cylinder functions (see, e.g., the surveys \cite{E,kerimov,watson} and references therein), significantly less inequalities and asymptotic results are known for zeros of $f_{\nu,R}$. Among known ones, the inequality of McCann \cite[(10)]{mccann} reads as
\begin{equation}\label{eq:mccan}
a_{\nu,k}
\geq 
\sqrt{a_{0,k}^2 + \frac{\nu^2}{R^2}},
\end{equation}
and the following  approximation of $a_{\nu,k}$ for a fixed $\nu$ of McMahon \cite[(24)]{mcmahon} (see also \cite[Theorem, p.~583]{cochran}) states that 
\begin{equation}\label{eq:a0>a}
a_{\nu,k} = \frac{\pi k}{R-1} + O\left(\frac{1}{k}\right).
\end{equation}

The aim of the present article is to characterize the asymptotic behavior of $a_{kx,k}$ as $k \to \infty$ for $x \geq 0$ and obtain an upper bound for $a_{\nu,k}$. To this end, we will use the result of Willis \cite[(8)]{willis} who derived the following formula for the derivative of $a_{\nu,k}$ with respect to $\nu$:
\begin{align}
\notag
\label{eq:willis}
\frac{da_{\nu,k}}{d\nu} 
&= \frac{2a_{\nu,k}(J_\nu^2(a_{\nu,k})+Y_\nu^2(a_{\nu,k}))\int_0^\infty K_0(2Ra_{\nu,k} \sinh t) e^{-2\nu t} \, dt}{(J_\nu^2(a_{\nu,k})+Y_\nu^2(a_{\nu,k})) - (J_\nu^2(Ra_{\nu,k})+Y_\nu^2(Ra_{\nu,k}))} \\[3mm]
&-\frac{2a_{\nu,k} (J_\nu^2(Ra_{\nu,k})+Y_\nu^2(Ra_{\nu,k}))\int_0^\infty K_0(2a_{\nu,k} \sinh t) e^{-2\nu t} \, dt}{(J_\nu^2(a_{\nu,k})+Y_\nu^2(a_{\nu,k})) - (J_\nu^2(Ra_{\nu,k})+Y_\nu^2(Ra_{\nu,k}))},
\end{align}
where $K_0$ is the modified Bessel function of the second kind and zero order.
In fact, Willis assumed that neither $J_{\nu}(a_{\nu,k})=J_\nu(Ra_{\nu,k})=0$ nor $Y_{\nu}(a_{\nu,k})=Y_\nu(Ra_{\nu,k})=0$ for a considered $a_{\nu,k}$. However, checking the derivation of \eqref{eq:willis}, it is easy to see that this assumption is redundant and \eqref{eq:willis} is valid for any $a_{\nu,k}$. For the convenience of the reader, we give corresponding arguments in Appendix \ref{sec:appendix} below.

Note that the denominators in \eqref{eq:willis} are positive since $J_\nu^2(z)+Y_\nu^2(z)$ decreases \cite[p.~446]{watson}.
By working with the numerators in \eqref{eq:willis}, Willis proved that $\frac{da_{\nu,k}}{d\nu} > 0$ for any $\nu > 0$, that is, positive zeros of $f_{\nu,R}$ increase with respect to $\nu \geq 0$; see also \cite[Section 5]{LM} for further results in this direction.

Let us state our main result.
\begin{theorem}\label{thm1}
	Let $R> 1$. Then 
	\begin{equation}\label{eq:converge}
	\lim_{k \to \infty} \frac{a_{kx,k}}{k} = \alpha(x),
	\quad x \geq 0 ,
	\end{equation}
	where $\alpha(x)$ is a unique solution of the initial value problem
	\begin{equation}\label{eq:main_diff}
	\frac{dy}{dx} = 
	\frac{\arccos\left(\frac{x}{Ry}\right) - \mathrm{Ac}\left(\frac{x}{y}\right)}{R\,\sqrt{1 - \left(\frac{x}{Ry}\right)^2} - \mathrm{Sr}\bigg(1-\Big(\frac{x}{y}\Big)^2\bigg)},
	\qquad 
	y(0) = \frac{\pi}{R-1}, 
	\end{equation}
	for $x \geq 0$.
	Here, the functions $\mathrm{Ac}$ and $\mathrm{Sr}$ are zero extensions of $\arccos$ and \textup{square root}, respectively, defined as
	\begin{equation}\label{eq:ArSq}
	\mathrm{Ac}(t) = 
	\left\{
	\begin{aligned}
	\arccos t  \quad \mathrm{for} \quad &|t| \leq 1,\\
	0 \quad \mathrm{for} \quad &|t| > 1,
	\end{aligned}
	\right.
	\qquad 
	\mathrm{Sr}(t) = 
	\left\{
	\begin{aligned}
	\sqrt{t} \quad \mathrm{for} \quad &t \geq 0,\\
	0 \quad \mathrm{for} \quad &t < 0.
	\end{aligned}
	\right.
	\end{equation}
\end{theorem}

We will study some basic properties of $\alpha(x)$ in Section \ref{sec:prop} below.

\medskip
Note that the differential equation \eqref{eq:willis} ``generalizes'' the formula \cite[p.~508]{watson}
\begin{equation}\label{eq:derivative_j}
\frac{dj_{\nu,\kappa}}{d\nu} = 2 j_{\nu,\kappa} \int_0^\infty K_0(2j_{\nu,\kappa} \sinh t) e^{-2\nu t} \, dt,
\end{equation}
where $j_{\nu,\kappa}$, $\kappa = k - \frac{\alpha}{\pi}$, $k=1,2,\dots,$ denotes the $k$-th positive zero of the cylinder function
$$
J_\nu(z) \cos \alpha - Y_\nu(z) \sin \alpha,
\quad
0 \leq \alpha < \pi.
$$
Using \eqref{eq:derivative_j}, Elbert and Laforgia \cite{EL1} proved that 
$$
\lim\limits_{\kappa \to \infty} \frac{j_{\kappa x, \kappa}}{\kappa} = \iota(x), \quad x>-1,
$$
where $\iota(x)$ is a unique solution of the initial value problem
\begin{equation*}\label{eq:el}
\frac{dy}{dx} = 
\frac{\arccos\left(\frac{x}{y}\right)}{\sqrt{1 - \left(\frac{x}{y}\right)^2}},
\qquad 
y(0) = \pi.
\end{equation*}
In fact, $\iota(x)$ admits a closed-form representation in terms of a solution of a transcendental equation \cite[(2.1)]{EL1}; see also \cite[Section 1.5]{E}. Our proof of Theorem \ref{thm1} is inspired by the approach of Elbert and Laforgia in combination with the formula \eqref{eq:willis} of Willis. 

\medskip

Theorem \ref{thm1} and the formula \eqref{eq:willis} can be helpful to obtain various bounds for $a_{\nu,k}$. In particular, we provide the following result.
\begin{theorem}\label{thm:upper}
	Let $R>1$, $\nu \geq 0$, and $k=1,2,\dots$
	Then the following upper bound is satisfied:
	\begin{equation}\label{eq:a:upper-bounds}
	a_{\nu,k} < \frac{\pi k}{R-1} + \frac{\pi \nu}{2 R}.
	\end{equation}
\end{theorem}

\medskip

Finally, we use Theorem \ref{thm1} to obtain a Pleijel-type result on the nodal domains statistics for the Laplace eigenfunctions in annuli. 
Denote by $\{\lambda_k\}$ the increasing sequence of eigenvalues of the Dirichlet eigenvalue problem
\begin{equation}\label{eq:dirichlet}
\left\{
\begin{aligned}
-\Delta u &= \lambda u &&{\rm in}\ \Omega, \\
u&=0 &&{\rm on }\ \partial \Omega,
\end{aligned}
\right.
\end{equation}
where $\Omega \subset \mathbb{R}^2$ is a bounded domain.
Let $\varphi_k$ be an eigenfunction associated with $\lambda_k$, and denote by $\mu(\varphi_k)$ the number of nodal domains of $\varphi_k$, that is, the number of connected components of the set $\Omega \setminus \overline{\{x\in\Omega: \varphi_k(x)=0\}}$. The famous nodal domain theorem of Courant asserts that $\mu(\varphi_k) \leq k$ for any $k$. Pleijel \cite{pleijel} obtained the following refinement of this fact:
\begin{equation*}\label{eq:pleijel}
Pl(\Omega) := \limsup_{k \to \infty} \frac{\mu(\varphi_k)}{k} \leq \frac{4}{j_{0,1}^2} = 0.69166\ldots
\end{equation*}
Here, $Pl(\Omega)$ is called \textit{Pleijel constant of} $\Omega$.
We refer the reader to the surveys \cite{BHHO,Helff} for the overview of results in this direction.

In connection with the conjecture of Polterovich \cite[Remark 2.2]{pol}, there is an interesting question to determine the exact value of the Pleijel constant $Pl(\Omega)$ for particular domains, see \cite[Section 6.1]{BHHO}.
In the article \cite{bobkov}, we investigated the values and expressions of $Pl(\Omega)$ for some symmetric domains like a disk, annuli (rings), and their sectors. 
In particular, under the notation 
$A_R := \{x \in \mathbb{R}^2:~ 1 < |x| < R\}$, $R > 1$, it was proved in \cite[Proposition 1.6]{bobkov} that 
\begin{equation}\label{eq:PLBa}
Pl(A_{R}) = \frac{8}{R^2-1} \, \sup_{x>0} \left\{x \, \limsup_{k \to \infty} \frac{k^2}{a_{kx,k}^2}\right\},
\end{equation}
provided any sufficiently large eigenvalue $\lambda_k$ of \eqref{eq:dirichlet} on $A_R$ has the multiplicity at most two. In the case of a higher multiplicity (which can possibly occur, see \cite[Lemma 1.9]{bobkov}), $Pl(A_R)$ is estimated by the right-hand side of \eqref{eq:PLBa} from below. 
Combining these facts with Theorem~\ref{thm1}, we obtain the following result.
\begin{prop}\label{prop:pleijel}
	Let $R>1$. Then 
	\begin{equation}\label{eq:PLBa1}
	Pl(A_{R}) \geq \frac{8}{R^2-1} \, \sup_{x>0} \left\{\frac{x}{\alpha(x)^2}\right\},
	\end{equation}	
	where $\alpha(x)$ is the solution of \eqref{eq:main_diff} defined in Theorem \ref{thm1} (see also Remark \ref{rem:analit}).
	The equality in \eqref{eq:PLBa1} is satisfied if any sufficiently large eigenvalue $\lambda_k$ of \eqref{eq:dirichlet} on $A_R$ has the multiplicity at most two.
\end{prop}

\medskip
The article is structured as follows. In Section \ref{sec:proof}, we prove Theorem \ref{thm1}. Section \ref{sec:prop} is devoted to the study of properties of $\alpha(x)$. In Section \ref{sec:upper-bound}, we prove Theorem \ref{thm:upper}. In Section \ref{sec:pleijel}, we provide some numerical results concerning the value of $Pl(A_R)$ for several $R>1$. Finally, for the convenience of the reader, we prove some basic properties of $f_{\nu,R}$ in Appendix \ref{sec:appendix}.

\section{Proof of Theorem \ref{thm1}}\label{sec:proof}

	Let us denote the right-hand side of the differential equation in \eqref{eq:main_diff} as $F$, that is, 
	\begin{equation*}\label{eq:F}
	F(x,y) = 
	\frac{\arccos\left(\frac{x}{Ry}\right) - \mathrm{Ac}\left(\frac{x}{y}\right)}{R\,\sqrt{1 - \left(\frac{x}{Ry}\right)^2} - \mathrm{Sr}\bigg(1-\Big(\frac{x}{y}\Big)^2\bigg)}.
	\end{equation*}
	Let us also introduce the set
	$$
	E = \left\{(x,y) \in \mathbb{R}^2:~ x \geq 0,~ y > \frac{x}{R} \right\},
	$$
	and split $E$ as $E = E_1 \cup E_2$, where
	$$
	E_1 = \left\{(x,y) \in \mathbb{R}^2:~ x > 0,~ \frac{x}{R}< y \leq x \right\},
	\quad 
	E_2 = \left\{(x,y) \in \mathbb{R}^2:~ x \geq 0,~ y > x \right\}.
	$$

	We start by showing that the initial value problem \eqref{eq:main_diff} possesses a unique solution $\alpha(x)$ for $x \geq 0$. 
	Noting that the functions $\mathrm{Ac}$ and $\mathrm{Sr}$ defined in \eqref{eq:ArSq} are continuous on $(-1,\infty)$ and $\mathbb{R}$, respectively, we see that $F$ is continuous on $E$. This fact yields the existence of a solution $\alpha(x)$ of \eqref{eq:main_diff} on a maximal interval $I = [0,b)$ for some $b \in (0, \infty]$. 
	Moreover, since $R>1$, $F$ is positive on $E$ for $x>0$, which implies that $\alpha(x)$ increases and hence $\alpha(x) > \alpha(0) = \frac{\pi}{R-1}$. 
	Using this fact and the inequality
	\begin{equation}\label{eq:F>1R}
	F(x,y) = \frac{\arccos\left(\frac{x}{Ry}\right)}{R\sqrt{1 - \left(\frac{x}{Ry}\right)^2}} 
	>
	\frac{1}{R}
	\quad \text{for } 
	(x,y) \in E_1,
	\end{equation}
	we see that $\alpha(x)$ cannot meet the boundary $y=\frac{x}{R}$ of $E$. 
	Let us now suppose that $\alpha(x) \to \infty$ as $x$ tends to some finite $x_0>0$. However, since $F(x,\alpha(x))$ stays finite, $\alpha(x)$ also stays finite on finite $x$-intervals, which leads to a contradiction. Thus, we conclude that $b=\infty$.
	To prove that $\alpha(x)$ is the unique solution of \eqref{eq:main_diff} for $x \geq 0$, we show that $F$ is one-sided Lipschitz with respect to $y$ provided $(x,y) \in E$ and $y > \varepsilon>0$ for some $\varepsilon>0$. More precisely, let us show that for any $\varepsilon >0$ and all $(x,y), (x,z) \in E$ with $y,z > \varepsilon$ there holds
	\begin{equation}\label{eq:lipschitz}
	\left(F(x,y)-F(x,z)\right)(y-z) \leq \frac{1}{R\varepsilon} (y-z)^2.
	\end{equation}
	Take any $\varepsilon>0$ and assume, without loss of generality, that $y<z$ for a pair $(x,y), (x,z) \in E$ with $y,z > \varepsilon$. 
	First, recalling \eqref{eq:F>1R}, we obtain for any $(x,t) \in E_1$ with $t>\varepsilon$ that
	\begin{align*}
	\left.\frac{\partial F(x,s)}{\partial s}\right|_{s=t} = 
	\frac{x}{R^2 t^2 - x^2} \left(1 - \frac{x F(x,t)}{t}\right)
	\leq \frac{x}{Rt(Rt+x)} \leq \frac{1}{R\varepsilon}.
	\end{align*}	
	Thus, if $(x,y), (x,z) \in E_1$, then, applying the mean value theorem, we deduce that $F$ satisfies \eqref{eq:lipschitz}.
	If $(x,y), (x,z) \in E_2$, then $F(x,y)>F(x,z)$, and hence \eqref{eq:lipschitz} is automatically satisfied. 
	The validity of \eqref{eq:lipschitz} in the remaining case $(x,y) \in E_1$, $(x,z) \in E_2$ follows by combining the previous two cases. 
	Indeed, by the mean value theorem there exists $t_0 \in [y,x]$ such that
	\begin{align*}
	&\left(F(x,y)-F(x,z)\right)(y-z)
	=
	\left(F(x,y)-F(x,x)+F(x,x)-F(x,z)\right)(y-z)\\
	&\leq
	\left(F(x,y)-F(x,x)\right)(y-z)
	= 
	\left.\frac{\partial F(x,s)}{\partial s}\right|_{s=t_0} (y-z)(y-x)
	\leq \frac{1}{R\varepsilon} (y-z)^2.
	\end{align*}
	Therefore, the standard uniqueness theorem (see, e.g.,  \cite[Chapter III, Theorem 6.1 and Exercise 6.8]{hartman}) implies that $\alpha(x)$ is the unique solution of \eqref{eq:main_diff}.

	Now we prove the convergence result \eqref{eq:converge}.
	Denote
	\begin{equation}\label{eq:bk}
	\alpha_k(x) = \frac{a_{kx,k}}{k}
	\quad \text{for }
	x \geq 0,
	~k =1,2,\dots
	\end{equation}
	Since $\left.\frac{d\alpha_{k}(x)}{dx} = \frac{da_{\nu,k}}{d\nu}\right|_{\nu=kx}$, we see from Willis' formula \eqref{eq:willis} that $\alpha_k(x)$ is the solution of the initial value problem
	\begin{equation}\label{eq:bk:diff}
	\frac{d y}{dx} = F_k(x,y), 
	\quad 
	y(0) = \alpha_k(0),
	\end{equation}
	where
	\begin{align}
	\notag
	F_k(x,y) 
	&= \frac{2 k y(J_{kx}^2(k y)+Y_{kx}^2(k y))\, 2 k y\int_0^\infty K_0(2Rk y \sinh t) e^{-2 kx t} \, dt}{2 k y (J_{kx}^2(k y)+Y_{kx}^2(k y)) - 2 k y (J_{kx}^2(Rk y)+Y_{kx}^2(Rk y))} \\[3mm]
	\label{eq:dbdx}
	&-\frac{2k y (J_{kx}^2(Rk y)+Y_{kx}^2(Rk y))\, 2 k y \int_0^\infty K_0(2k y \sinh t) e^{-2 kx t} \, dt}{2 k y(J_{kx}^2(k y)+Y_{kx}^2(k y)) - 2 k y(J_{kx}^2(Rk y)+Y_{kx}^2(Rk y))}.
	\end{align}	
	(The additional multipliers $2ky$ in \eqref{eq:dbdx} are included for the simplicity of further usage.)
	Moreover, $\alpha_k(x)$ is the unique solution of \eqref{eq:bk:diff} due to the continuous differentiability of $F_k$ with respect to $y>0$. 
	Below, we will expand $F_k$ further via Nicholson's formula	\cite[(1), p.~444]{watson}
	\begin{align}
	\notag
	J_{kx}^2(Rky) &+ Y_{kx}^2(Rky) =
	\frac{8}{\pi^2} \int_0^\infty K_0(2 Rky \sinh t) \cosh(2kx t) \, dt \\
	\label{eq:J+Y}
	&= \frac{4}{\pi^2} \int_0^\infty K_0(2 Rky \sinh t) e^{-2kx t} \, dt 
	+
	\frac{4}{\pi^2} \int_0^\infty K_0(2 Rky \sinh t) e^{2kx t} \, dt.
	\end{align}
	Note that the lower bound \eqref{eq:mccan} implies that $(x, \alpha_k(x)) \in E$ for any $x \geq 0$ and $k$. That is, we can assume that each $F_k$ is defined on $E$. 
	
	We are going to show that $F_k$ converges to $F$ uniformly on every compact subset of $E_1$ and $E_2$ as $k \to \infty$. Then \cite[Chapter II, Theorem 3.2]{hartman} in combination with the uniqueness of $\alpha(x)$ obtained above will imply that $\alpha_k(x) \to \alpha(x)$ for each $x \geq 0$, which is equivalent to the desired result \eqref{eq:converge}.
	To this end, we will show the local uniform convergence and local boundedness of integrals in \eqref{eq:dbdx} and \eqref{eq:J+Y}.
	
	We start by performing the following trivial change of variables:
	\begin{equation}\label{eq:K0-K0}
	2ky\int_0^\infty K_0(2Rky \sinh t) e^{-2 kx t} \, dt 
	= 
	\frac{1}{R}\int_0^\infty K_0\left(u \, \frac{\sinh\left(\frac{u}{2Rky}\right)}{\frac{u}{2Rky}}\right) e^{-\frac{x}{Ry} u} \, du.
	\end{equation}
	Since $\sinh t > t$ for $t>0$, $\frac{\sinh t}{t} \to 1$ as $t \to 0$, and $K_0(t)$ is decreasing and integrable, we obtain that
	\begin{equation}\label{eq:k0r-upper}
	2ky\int_0^\infty K_0(2Rky \sinh t) e^{-2 kx t} \, dt 
	< \frac{1}{R}\int_0^\infty K_0(u) e^{-\frac{x}{R y} u} \, du \leq C < \infty
	\end{equation}	
	for all $(x,y) \in E$ and $k$, where $C>0$ is a uniform constant.
	To obtain a lower bound for \eqref{eq:K0-K0}, we use a Cusa-Huygens-type inequality from \cite{NS} which reads as
	$$
	\frac{\sinh\left(\frac{u}{2Rky}\right)}{\frac{u}{2Rky}} < \frac{\cosh\left(\frac{u}{2Rky}\right) + 2}{3}.
	$$
	Let us introduce $u_k > 0$ such that 
	$$
	\frac{\cosh\left(\frac{u_k}{2Rky}\right) + 2}{3} = 1 + \frac{1}{k}.
	$$
	Using the series expansion of $\mathrm{arccosh}\, t$, we get
	$$
	u_k = 2Rky \,\mathrm{arccosh}\left(1+\frac{3}{k}\right) = y \, O(\sqrt{k}) \to \infty
	\quad 
	\text{as }
	k \to \infty.
	$$
	Therefore, using the monotonicity of $K_0(t)$ and $\cosh t$ for $t>0$, we deduce that
	\begin{equation}\label{eq:K0>K0}
	\int_0^\infty K_0\left(u \, \frac{\sinh\left(\frac{u}{2Rky}\right)}{\frac{u}{2Rky}}\right) e^{-\frac{x}{Ry} u} \, du \geq 
	\int_0^{u_k} K_0\left(u \left(1+\frac{1}{k}\right)\right) e^{-\frac{x}{Ry} u} \, du.
	\end{equation}
	Since by \cite[(1), p.~202]{watson}, 
	\begin{equation}\label{eq:K0est}
	K_0(t) \sim \sqrt\frac{\pi}{2t} \, e^{-t} 
	\quad
	\text{as } t \to \infty,
	\end{equation}
	we see that
	\begin{equation*}
	\int_{u_k}^\infty K_0\left(u \left(1+\frac{1}{k}\right)\right) e^{-\frac{x}{Ry} u} \, du \to 0
	\quad
	\text{as } k \to \infty
	\end{equation*}
	locally uniformly on $E$.
	(Note that the locality comes from the fact that $u_k$ depends on $y$.)
	Thus, 
	$$
	\int_0^{u_k} K_0\left(u \left(1+\frac{1}{k}\right)\right) e^{-\frac{x}{Ry} u} \, du 
	\to 
	\int_0^\infty K_0(u) e^{-\frac{x}{R y} u} \, du 
	\quad
	\text{as } k \to \infty
	$$
	and hence
	\begin{equation}\label{eq:conv1}
	2ky\int_0^\infty K_0(2Rky \sinh t) e^{-2 kx t} \, dt 
	\to \frac{1}{R}\int_0^\infty K_0(u) e^{-\frac{x}{R y} u} \, du 
	\quad
	\text{as } k \to \infty
	\end{equation}
	locally uniformly on $E$, and the left-hand side of \eqref{eq:conv1} is bounded on $E$, see \eqref{eq:k0r-upper}. 
	Clearly, the same results hold true for the integral $2ky\int_0^\infty K_0(2ky \sinh t) e^{-2 kx t} \, dt$.
	
	Analogously, recalling that $\frac{x}{Ry} < 1$ for $(x,y) \in E$, 
	we deduce from \eqref{eq:K0est} that 
	\begin{equation*}
	\int_{u_k}^\infty K_0\left(u \left(1+\frac{1}{k}\right)\right) e^{\frac{x}{Ry} u} \, du \to 0
	\quad
	\text{as } k \to \infty
	\end{equation*}
	and hence
	\begin{equation}\label{eq:conv3}
	2ky\int_0^\infty K_0(2Rky \sinh t) e^{2 kx t} \, dt 
	\to \frac{1}{R}\int_0^\infty K_0(u) e^{\frac{x}{R y} u} \, du 
	\quad
	\text{as } k \to \infty
	\end{equation}
	locally uniformly on $E$. 
	Moreover, the right-hand side of \eqref{eq:conv3} is locally bounded on $E$, and the same holds true for the left-hand side, cf.\ \eqref{eq:k0r-upper}.
	
	In its turn, the treatment of the remaining integral $2ky\int_0^\infty K_0(2ky \sinh t) e^{2 kx t} \, dt$ already depends on the choice of $E_1$ or $E_2$.
	Indeed, since $\frac{x}{y}<1$ on $E_2$, we see, as above, that 
	\begin{equation}\label{eq:conv41}
	2ky\int_0^\infty K_0(2ky \sinh t) e^{2 kx t} \, dt 
	\to \int_0^\infty K_0(u) e^{\frac{x}{y} u} \, du 
	\quad
	\text{as } k \to \infty
	\end{equation}
	locally uniformly on $E_2$, and the right-hand side (and hence the left-hand side) of \eqref{eq:conv41} is locally bounded away from zero and infinity on $E_2$.
	On the other hand, using \eqref{eq:K0est}, we deduce as in \eqref{eq:K0>K0} that 
	\begin{equation}\label{eq:conv42}
	2ky\int_0^\infty K_0(2ky \sinh t) e^{2 kx t} \, dt 
	\geq
	\int_0^{u_k} K_0\left(u \left(1+\frac{1}{k}\right)\right) e^{\frac{x}{y} u} \, du
	\to \infty
	\end{equation}
	as $k \to \infty$ locally uniformly on $E_1$. 
	
	Let us now rewrite $F_k$ as $F_k = \frac{G_k}{H_k}$, where
	\begin{align*}
	G_k(x,y) 
	&:= 2 k y \int_0^\infty K_0(2Rk y \sinh t) e^{-2 kx t} \, dt \\
	&- \frac{2 k y \int_0^\infty K_0(2Rk y \sinh t) e^{2 kx t} \, dt \cdot 2 k y \int_0^\infty K_0(2k y \sinh t) e^{-2 kx t} \, dt}{2 k y \int_0^\infty K_0(2k y \sinh t) e^{2 kx t} \, dt}
	\end{align*}
	and
	\begin{align*}
	H_k(x,y)
	&:= 1 + \frac{2 k y \int_0^\infty K_0(2k y \sinh t) e^{-2 kx t} \, dt}{2 k y \int_0^\infty K_0(2k y \sinh t) e^{2 kx t} \, dt} 
	\\
	&- \frac{2 k y \int_0^\infty K_0(2Rk y \sinh t) e^{2 kx t} \, dt}{2 k y \int_0^\infty K_0(2k y \sinh t) e^{2 kx t} \, dt}
	- \frac{2 k y \int_0^\infty K_0(2Rk y \sinh t) e^{-2 kx t} \, dt}{2 k y \int_0^\infty K_0(2k y \sinh t) e^{2 kx t} \, dt}.
	\end{align*}
	Recall that if some functional sequences $\{g_k\}$ and $\{h_k\}$ are locally bounded and locally uniformly convergent to $g$ and $h$, respectively, then $g_k h_k \to g h$ locally uniformly. Using this fact and the local uniform convergence and local boundedness of the integrals in \eqref{eq:conv1}, \eqref{eq:conv3}, and \eqref{eq:conv41} or \eqref{eq:conv42}, we see that
	$$
	F_k(x,y) \to \frac{1}{R}\int_0^\infty K_0(u) e^{-\frac{x}{R y} u} \, du
	\quad
	\text{as } k \to \infty
	$$
	locally uniformly on $E_1$, and $F_k \to \frac{G}{H}$ locally uniformly on $E_2$, where
	\begin{align*}
	G(x,y) 
	:= \frac{1}{R}\int_0^\infty K_0(u) e^{-\frac{x}{R y} u} \, du
	- \frac{\frac{1}{R}\int_0^\infty K_0(u) e^{\frac{x}{R y} u} \, du \cdot \int_0^\infty K_0(u) e^{-\frac{x}{y} u} \, du}{\int_0^\infty K_0(u) e^{\frac{x}{y} u} \, du}
	\end{align*}
	and
	\begin{align*}
	H(x,y)
	:= 1 
	+ \frac{\int_0^\infty K_0(u) e^{-\frac{x}{y} u} \, du}{\int_0^\infty K_0(u) e^{\frac{x}{y} u} \, du} 
	- \frac{\frac{1}{R}\int_0^\infty K_0(u) e^{\frac{x}{Ry} u} \, du}{\int_0^\infty K_0(u) e^{\frac{x}{y} u} \, du}
	- \frac{\frac{1}{R}\int_0^\infty K_0(u) e^{-\frac{x}{Ry} u} \, du}{\int_0^\infty K_0(u) e^{\frac{x}{y} u} \, du}.
	\end{align*}
	Simplifying the above expressions via the formula \cite[p.~388]{watson}
	\begin{equation}\label{eq:watson}
	\int_0^\infty K_0(u) e^{-a u} \, du = \frac{\arccos a}{\sqrt{1-a^2}},
	\quad 
	|a| < 1,
	\end{equation}	
	we conclude that $F_k$ converges locally uniformly on $E_1$ and $E_2$ to $F$.
	
	\smallskip
	As an auxiliary fact which will be used also in Section \ref{sec:prop} below, let us note that the following inequality is satisfied:
	\begin{equation}\label{eq:Fk:upper}
	F_k(x,y) < \frac{\arccos\left(\frac{x}{Ry}\right)}{R\sqrt{1 - \left(\frac{x}{Ry}\right)^2}},
	\quad (x,y) \in E,
	~k=1,2,\ldots
	\end{equation}
	To obtain \eqref{eq:Fk:upper}, it is enough to recall that $K_0(t)$ decreases for $t>0$, which implies that 
	\begin{equation*}\label{eq:K0R<K0}
	\int_0^\infty K_0(2k y \sinh t) e^{-2 kx t} \, dt 
	>
	\int_0^\infty K_0(2Rk y \sinh t) e^{-2 kx t} \, dt
	\end{equation*}
	for any $y>0$ and $k=1,2,\ldots$ 
	Recalling also that the denominators in \eqref{eq:dbdx} are positive, we deduce from \eqref{eq:dbdx} that
	$$
	F_k(x,y) < 2ky \int_0^\infty K_0(2 R k y \sinh t) e^{-2 kx t} \, dt.
	$$
	Therefore, using \eqref{eq:k0r-upper} and \eqref{eq:watson}, we obtain \eqref{eq:Fk:upper}. 
	\smallskip
	
	Finally, let us show that the obtained local uniform convergence of $F_k$ to $F$ implies the convergence result \eqref{eq:converge}. 
	Note that $(0,\alpha_k(0)) \in E_2$ for all $k$, and  \eqref{eq:a0>a} for $\nu=0$ reads as $\alpha_k(0) \to \alpha(0)$.
	Therefore, applying \cite[Chapter II, Theorem 3.2]{hartman} (with minor modifications) on $E_2$ and recalling that $\alpha(x)$ is the unique solution of \eqref{eq:main_diff}, we deduce that $\alpha_k(x) \to \alpha(x)$ for any $x \in [0,x_0)$, where $x_0$ defines the right maximal interval of applicability of \cite[Chapter II, Theorem 3.2]{hartman}. If $x_0 = \infty$, then we are done. If $x_0 < \infty$, then the only possibility is that $x_0 = \alpha(x_0)$, that is, $(x_0, \alpha(x_0)) \in \partial E_2 \cap \partial E_1$. 
	Since 
	$\arccos t = \arctan \left(\frac{\sqrt{1-t^2}}{t}\right) < \frac{\sqrt{1-t^2}}{t}$ for $t \in (0,1)$, we get
	\begin{equation}\label{eq:F>1R2}
	F(x,y) = \frac{\arccos\left(\frac{x}{Ry}\right)}{R\sqrt{1 - \left(\frac{x}{Ry}\right)^2}} 
	< \frac{y}{x} \leq 1
	\quad \text{for } 
	(x,y) \in E_1.
	\end{equation}
	Thus, we see that $\alpha(x) < x$ and hence $(x, \alpha(x)) \in E_1$ for all $x>x_0$.
	Moreover, in view of the inequalities \eqref{eq:Fk:upper} and \eqref{eq:F>1R2}, $\alpha_k(x) < x$ for any $x>x_0$ and sufficiently large $k$. 
	Fixing $x_1>x_0$ and applying \cite[Chapter II, Theorem 3.2]{hartman} on $E_1$, we can extract a subsequence $\{\alpha_{k_n}(x)\}$ which converges locally uniformly on the maximal interval $(x_0, x_2)$ to a solution $\tilde{\alpha}(x)$ of the differential equation in  \eqref{eq:main_diff} with the initial value $\tilde{\alpha}(x_1) = \lim\limits_{n\to \infty}\alpha_{k_n}(x_1)$. By continuity, $\alpha(x_0) = \tilde{\alpha}(x_0)$, and the uniqueness of $\alpha(x)$ yields $\alpha(x) = \tilde{\alpha}(x)$ for all $x \in (x_0,x_2)$. Moreover, $x_2=\infty$, as it follows from \eqref{eq:F>1R2} and \eqref{eq:F>1R}. 
	Thus, we conclude that $\alpha_k(x) \to \alpha(x)$ for any $x \geq 0$.	
\qed

\section{Properties of \texorpdfstring{$\alpha(x)$}{a(x)}}\label{sec:prop}

We start with auxiliary results which will be used to obtain upper bounds for $\alpha(x)$.
We use the notations $F_k$, $F$, and $E$ from Section \ref{sec:proof}. 
Let us note that \eqref{eq:Fk:upper} and the local uniform convergence of $F_k$ to $F$ proved in Section \ref{sec:proof} imply
\begin{equation}\label{eq:Fk:upper2}
F(x,y) \leq \frac{\arccos\left(\frac{x}{Ry}\right)}{R\sqrt{1 - \left(\frac{x}{Ry}\right)^2}},
\quad (x,y) \in E.
\end{equation}

Let now $\tilde{\iota}(x)$ be a unique solution of the initial value problem
\begin{equation}\label{eq:ivp}
\frac{dy}{dx} = 
\frac{\arccos\left(\frac{x}{y}\right)}{\sqrt{1 - \left(\frac{x}{y}\right)^2}},
\qquad 
y(0) = \frac{\pi R}{R-1},
\end{equation}
for $x > 0$. 
The uniqueness of $\tilde{\iota}(x)$ follows from the fact that the right-hand side of the differential equation in \eqref{eq:ivp} is Lipschitz provided $y > \varepsilon$ for some $\varepsilon > 0$. 
Note that, $\tilde{\iota}(x)$ can be expressed as a solution of the following system (see also \cite[(2.2) and (2.10)]{EL1}):
$$
\tilde{\iota}(x) = \frac{x}{\sin \alpha},
\qquad
\frac{\sin \alpha}{\cos \alpha - (\pi/2-\alpha) \sin \alpha} = \frac{(R-1)x}{\pi R},
\qquad
\alpha \in \left(0, \frac{\pi}{2}\right),
$$
where the second equation can be equivalently written as
\begin{equation}\label{eq:alpha}
\cot \alpha - \left(\frac{\pi}{2}-\alpha\right) = \frac{\pi R}{(R-1)x}.
\end{equation}
The left-hand side of \eqref{eq:alpha} tends to $+\infty$ as $\alpha \to 0+$, tends to $0$ as $\alpha \to \pi/2$, and decreases on $(0, \pi/2)$. Thus, for any $x > 0$ there exists a unique $\alpha \in (0,\pi/2)$ satisfying \eqref{eq:alpha}, and hence $\tilde{\iota}(x)$ is also determined. Moreover, in view of the monotonicity of the right-hand side of \eqref{eq:alpha}, we conclude that $\alpha = \alpha(x)$ is increasing, which yields $\alpha'(x) \geq 0$.

It is not hard to see that $\tilde{\iota}(x)$ is increasing and concave. Indeed, substituting $\tilde{\iota}(x) = \frac{x}{\sin \alpha(x)}$ into \eqref{eq:ivp}, we get
$$
\frac{d\tilde{\iota}(x)}{dx} = \frac{\pi/2-\alpha(x)}{\cos \alpha(x)} > 0.
$$
Taking the second derivative of $\tilde{\iota}(x)$ and recalling that $\alpha'(x) \geq 0$, we easily conclude that $\tilde{\iota}''(x)$ is negative, and hence $\tilde{\iota}(x)$ is concave.
By the concavity, 
\begin{equation}\label{eq:i:upper}
\tilde{\iota}(x) \leq \tilde{\iota}(0) + \tilde{\iota}'(0) x 
= 
\frac{\pi R}{R-1} + \frac{\pi x}{2}.
\end{equation}

Now we are ready to collect some basic properties of $\alpha(x)$. 
\begin{prop}\label{prop:prop}
	Let $R>1$ and let $\alpha(x)$ be the solution of \eqref{eq:main_diff}. Then $\alpha(x)$ is increasing and
	\begin{equation}\label{eq:b:estimate}
	\sqrt{\frac{\pi^2}{(R-1)^2} + \frac{x^2}{R^2}} < \alpha(x) < \frac{\tilde{\iota}(x)}{R} \leq \frac{\pi}{R-1} + \frac{\pi x}{2 R},
	\quad
	x>0,
	\end{equation}
	where $\tilde{\iota}(x)$ is the solution of \eqref{eq:ivp}. 
\end{prop}
\begin{proof}
	The monotonicity of $\alpha(x)$ is a consequence of the positivity of $F$ on $E$. Recalling that the right-hand side of \eqref{eq:Fk:upper2} is Lipschitz provided $y > \varepsilon$ for some $\varepsilon > 0$, the first upper bound in \eqref{eq:b:estimate} follows from \eqref{eq:Fk:upper2} by noting that for sufficiently small $x \geq 0$ the inequality in \eqref{eq:Fk:upper2} is strict. The last inequality in \eqref{eq:b:estimate} follows from \eqref{eq:i:upper}.
	
	To obtain a lower bound for $\alpha(x)$ we note that
	$$
	\arccos\left(\frac{z}{R}\right) - \arccos\left(z\right)
	=
	\arccos\left(\frac{z^2}{R}+ \sqrt{1 - \left(\frac{z}{R}\right)^2}\sqrt{1 - z^2}\right), 
	\quad 
	0<z<1.
	$$
	Applying now the inequality $\arccos t > \sqrt{1-t^2}$ for $t \in [0,1)$, we get
	\begin{equation}\label{eq:F'>1/R}
	\frac{\arccos\left(\frac{z}{R}\right) - \arccos\left(z\right)}{R\,\sqrt{1 - \left(\frac{z}{R}\right)^2} - \sqrt{1 - z^2}} > \frac{z}{R} > \frac{z}{R^2},
	\quad 
	0<z<1.
	\end{equation}
	On the other hand,
	\begin{equation}\label{eq:F:asympt}
	\frac{\arccos\left(\frac{z}{R}\right)}{R\,\sqrt{1 - \left(\frac{z}{R}\right)^2}} \to \frac{1}{R}
	\quad 
	\text{as }
	z \to R-,
	\end{equation}
	and the left-hand side of \eqref{eq:F:asympt} decreases on $(1,R)$.
	Therefore, we deduce that
	$$
	\frac{\arccos\left(\frac{z}{R}\right) - \mathrm{Ac}\left(z\right)}{R\,\sqrt{1 - \left(\frac{z}{R}\right)^2} - \mathrm{Sr}\left(1-z^2\right)} > \frac{z}{R^2},
	\quad 
	0 < z < R,
	$$
	which implies that $F(x,y) > \frac{x}{R^2 y}$ for $(x,y) \in E$, $x>0$, and hence the lower bound for $\alpha(x)$ in \eqref{eq:b:estimate} follows.
\end{proof}

\begin{remark}\label{rem:analit}
	The value of $\alpha = \alpha(x)$ can be also found as a unique solution of the transcendental equation
	\begin{equation}
	\label{eq:analit_1}
	R \alpha \sqrt{1 - \left(\frac{x}{R \alpha}\right)^2} - \alpha \sqrt{1 - \left(\frac{x}{\alpha}\right)^2}
	=
	x \arccos \left(\frac{x}{R \alpha}\right) - x \arccos \left(\frac{x}{\alpha}\right) + \pi
	\end{equation}
	for $x \in \left[0, x_0\right]$, and
	\begin{equation}
	\label{eq:analit_2}
	R \alpha \sqrt{1 - \left(\frac{x}{R \alpha}\right)^2} - x_0 \sqrt{R^2-1} 
	=
	x \arccos\left(\frac{x}{R \alpha}\right) - x_0 \arccos\left(\frac{1}{R}\right)
	\end{equation}
	for $x \in \left[x_0, +\infty\right)$, 
	where
	$$
	x_0 := \frac{\pi}{\sqrt{R^2-1} -\arccos\left(1/R\right)}.
	$$
	In fact, knowing that the value $\lim\limits_{k \to \infty} \frac{a_{kx,k}}{k}$ exists by Theorem \ref{thm1}, one can obtain \eqref{eq:analit_1} using the leading terms of trigonometric Debye asymptotics of Bessel functions $J_\nu$ and $Y_\nu$ (see \cite[pp.~244-245]{watson}). On the other hand, since $\alpha(x_0) = x_0$, we have $\alpha(x) < x$ for all $x > x_0$ (see \eqref{eq:F>1R2}), and the equation \eqref{eq:analit_2} can be obtained by integrating \eqref{eq:main_diff} using, e.g., the substitute $y = \frac{x}{\sin \gamma}$; see \cite[(2.2) and (2.11)]{EL1}.
\end{remark}

\section{The upper bound for \texorpdfstring{$a_{\nu,k}$}{a}}\label{sec:upper-bound}
Let us turn to the proof of the upper bound \eqref{eq:a:upper-bounds} of Theorem \ref{thm:upper}. 
To this end, consider the eigenvalue problem
\begin{equation}\label{eq:diff:radial}
\left\{
\begin{aligned}
-(r u')' &= \mu r u, \quad r \in (1,R), \\
u(1) &= u(R) =0.
\end{aligned}
\right.
\end{equation}
It can be checked that the $k$-th eigenvalue $\mu_k$ of \eqref{eq:diff:radial} is equal to $a_{0,k}^2$, and
$$
\psi_{k}(r) = J_0(a_{0,k} r) Y_0(a_{0,k}) - J_0(a_{0,k}) Y_0(a_{0,k} r)
$$
is the unique (modulo scaling) eigenfunction associated with $\mu_k$. 
Moreover, $\psi_k$ has exactly $k$ nodal domains. 
Let us denote $r_0=1$ and $r_k=R$, for convenience. Then $\mu_k$ can be characterized as
\begin{equation}\label{eq:mu:def}
\mu_k = \min\limits_{r_0<r_1<\dots<r_{k-1}<r_k} \max \left\{
\mu_1^{(r_0,r_1)}, \dots,
\mu_1^{(r_{k-1},r_k)}
\right\},
\end{equation}
see \cite[Lemma 2.2]{BD} with minor modifications.
Here $\mu_1^{(r_i,r_{i+1})}$ is the first eigenvalue of \eqref{eq:diff:radial} on $(r_i,r_{i+1})$, that is,
$$
\mu_1^{(r_i,r_{i+1})} = \inf_{u \in W_0^{1,2}(r_i,r_{i+1}) \setminus \{0\}} \frac{\int_{r_i}^{r_{i+1}} r |u'|^2 \, dr}{\int_{r_i}^{r_{i+1}} r |u|^2 \, dr}.
$$

The proof of Theorem \ref{thm:upper} will be based on the upper estimate \eqref{eq:Fk:upper} for $F_k$ and the following fact.
\begin{prop}\label{prop:a0-upper-bound}
	Let $R>1$. Then for any $k=1,2,\ldots$ the following inequality is satisfied:
	\begin{equation}\label{eq:a0-upper-bound}
	a_{0,k} < \frac{\pi k}{R-1}.
	\end{equation}
\end{prop}
\begin{proof}
	We use the characterization \eqref{eq:mu:def} to estimate $\mu_k = a_{0,k}^2$ from above. 
	As an admissible function for each $\mu_1^{(r_i,r_{i+1})}$ we use
	$$
	v(r) = \cos \left(\frac{\pi k r}{R-1}\right) - \cot \left(\frac{\pi k}{R-1}\right) \sin \left(\frac{\pi k r}{R-1}\right)
	$$
	whenever $\frac{k}{R-1} \not\in \mathbb{N}$, and 
	$$
	v(r) = \tan \left(\frac{\pi k}{R-1}\right)\cos \left(\frac{\pi k r}{R-1}\right) - \sin \left(\frac{\pi k r}{R-1}\right)
	$$
	otherwise.
	Denoting $r_i = 1 + \frac{(R-1)i}{k}$, we see that $v \in W_0^{1,2}(r_i, r_{i+1}) \setminus \{0\}$ and 
	\begin{equation}\label{eq:mu1leq}
	\mu_1^{(r_i,r_{i+1})} \leq \frac{\int_{r_i}^{r_{i+1}} r |v'|^2 \, dr}{\int_{r_i}^{r_{i+1}} r |v|^2 \, dr} 
	= 
	\frac{\pi^2 k^2}{(R-1)^2}, \quad i=0,\dots,k-1.
	\end{equation}
	Recalling that $\psi_k$ satisfies \eqref{eq:diff:radial} and $v$ satisfies  $-v'' = \frac{\pi^2 k^2}{(R-1)^2} v$, we deduce that $\psi_k$ and $v$ are linearly independent on any $(r_i,r_{i+1})$.
	Thus, since each $\mu_1^{(r_i,r_{i+1})}$ possesses a unique minimizer (modulo scaling), we conclude that the strict inequality in \eqref{eq:mu1leq} holds true, and hence \eqref{eq:mu:def} implies \eqref{eq:a0-upper-bound}.
\end{proof}

	Let us now prove Theorem \ref{thm:upper}.
\begin{proof}[Proof of Theorem \ref{thm:upper}]
	Under the notation \eqref{eq:bk}, the inequality \eqref{eq:a0-upper-bound} reads as $\alpha_k(0) < \frac{\pi}{R-1}$. Using this fact and the inequality \eqref{eq:Fk:upper}, we conclude that 
	\begin{equation*}
	\alpha_k(x) < \frac{\tilde{\iota}(x)}{R}
	\end{equation*}
	where $\tilde{\iota}(x)$ is the solution of \eqref{eq:ivp},
	which reads as
	\begin{equation}\label{eq:bk:upper-estimate0}
	a_{\nu,k} < \frac{\tilde{\iota}(\nu/k)k}{R}.
	\end{equation}
	Applying the inequality \eqref{eq:i:upper}, we obtain the desired upper bound \eqref{eq:a:upper-bounds}.
\end{proof}

\begin{remark}
	One can derive other upper bounds for $a_{\nu,k}$ from \eqref{eq:bk:upper-estimate0} by estimating $\tilde{\iota}(x)$ from above in a different way than \eqref{eq:i:upper}.
\end{remark}

\section{Pleijel's constant for annuli}\label{sec:pleijel}

For each fixed $R>1$, one can numerically solve the initial value problem \eqref{eq:main_diff} (or equations \eqref{eq:analit_1} and \eqref{eq:analit_2}) to obtain $\alpha(x)$ and hence compute the value
\begin{equation}\label{eq:Pl}
\frac{8}{R^2-1} \, \sup_{x>0} \left\{\frac{x}{\alpha(x)^2}\right\}
\end{equation}
from Proposition \ref{prop:pleijel}. 
Using a build-in ODE-solver of \textsl{Mathematica}, we obtained several approximate values of \eqref{eq:Pl} listed in Table \ref{tab}. Recall that if the multiplicity of any sufficiently large eigenvalue of \eqref{eq:dirichlet} on $A_R$ is at most two, then \eqref{eq:Pl} gives an exact value of the Pleijel constant $Pl(A_R)$. 

\begin{table}[!h]
	\centering
	\begin{tabular}{| c || c | c | c | c | c | c | c |}
		\hline
		$R = $ 			& 1.05	& 1.1	& 1.5	& 2	& 4 & 6 & 10 \\
		\hline
		$Pl(A_R) \geq $ & 0.636367 & 0.635656	& 0.619308 & 0.58654 & 0.492055 & 0.474482 & 0.465961 \\
		\hline
	\end{tabular}
	\caption{Several approximate values of \eqref{eq:Pl}.}
	\label{tab}
\end{table}
The values from Table \ref{tab} suggest that the following asymptotics are satisfied (see \cite[Remark 1.8]{bobkov}):
\begin{align*}
\lim_{R \to 1} \left[\frac{8}{R^2-1} \, \sup_{x>0} \left\{\frac{x}{\alpha(x)^2}\right\}\right] 
&= 
Pl(\mathcal{R}) = \frac{2}{\pi} = 0.6366197\ldots, \\
\lim_{R \to \infty} \left[\frac{8}{R^2-1} \, \sup_{x>0} \left\{\frac{x}{\alpha(x)^2}\right\}\right] 
&= 
Pl(B) = 0.4613019\ldots,
\end{align*}
and \eqref{eq:Pl} lies in between these two values. 
Here $\mathcal{R}$ is any rectangle $(0,a) \times (0,b)$ with irrational ratio $\frac{a^2}{b^2}$ (see \cite[Proposition 5.1]{Helff}), and $B$ is a planar disk (see \cite[Theorem 1.3]{bobkov}).

It is also informative to compare the values from Table \ref{tab} with the behavior of the ratio $\frac{\mu(\varphi_k)}{k}$ for large $k$, where $\varphi_k$ is the $k$-th eigenfunction of \eqref{eq:dirichlet} on $A_R$ of the form
\begin{equation*}\label{eq:eigenfunction}
\varphi_n(r,\theta) = \left(J_\nu(a_{\nu,n} r) Y_\nu(a_{\nu,n}) - J_\nu(a_{\nu,n}) Y_\nu(a_{\nu,n} r) \right)\cos(\nu \theta),
\end{equation*}
for some $n = 1,2,\dots$ and $\nu = 0,1,\dots$, where $r \in (1,R)$ and $\theta \in [0,2\pi)$. 
We present the corresponding plots on Figures \ref{fig2} and \ref{fig3}.

\begin{figure}[ht]
	\centering
	\includegraphics[width=0.9\linewidth]{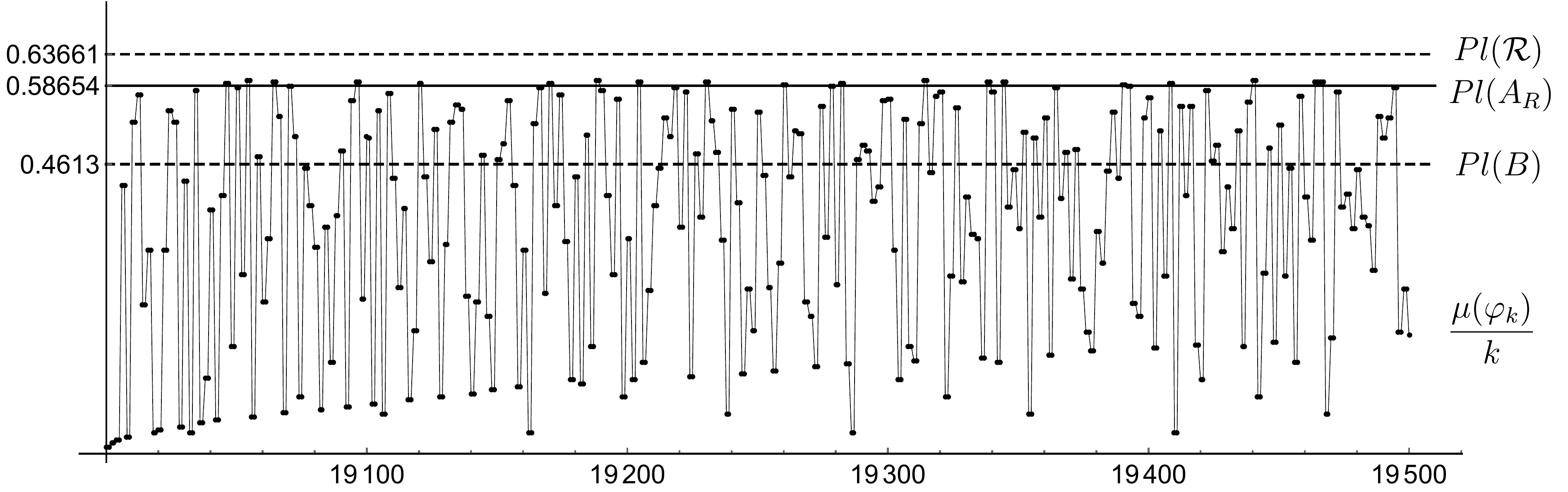}\\
	\caption{$R=2$. The values of $\frac{\mu(\varphi_k)}{k}$ for $k = 19000,\dots, 19500$.}
	\label{fig2}
\end{figure}

\begin{figure}[ht]
	\centering
	\includegraphics[width=0.9\linewidth]{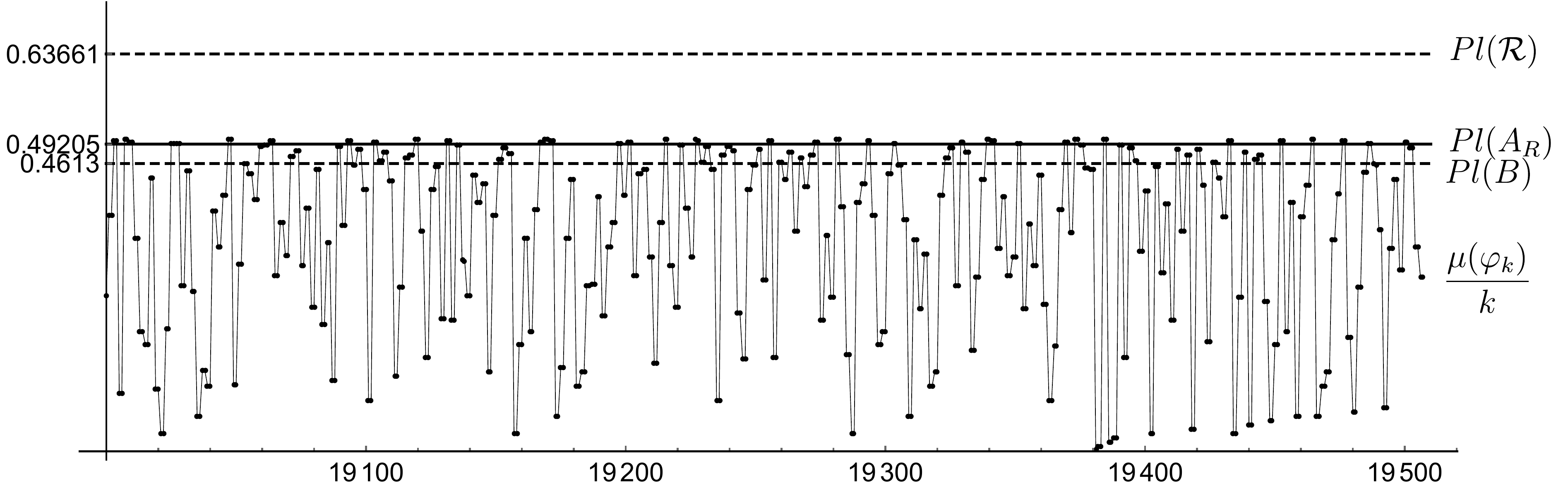}\\
	\caption{$R=4$. The values of $\frac{\mu(\varphi_k)}{k}$ for $k = 19000,\dots, 19500$.}
	\label{fig3}
\end{figure}

\appendix
\section{Appendix}\label{sec:appendix}

\begin{lemma}
	Let $\nu,z \in \mathbb{R}$ and $R>0$. Then $f_{\nu,R}(z)$ is even with respect to $\nu$ and $z$.
\end{lemma}
\begin{proof}
	Evenness of $f_{\nu,R}(z)$ with respect to $\nu$ follows by applying the equalities
	$$
	J_{-\nu}(z) = (-1)^\nu J_\nu(z)
	\quad 
	\text{and}
	\quad
	Y_{-\nu}(z) = (-1)^\nu Y_\nu(z),
	\quad
	\text{provided }
	\nu \in \mathbb{Z},
	$$
	and from the relation
	$$
	Y_\nu(z) = \frac{J_\nu(z) \cos (\nu \pi) - J_{-\nu}(z)}{\sin (\nu \pi)},
	\quad 
	\text{provided }
	\nu \not\in \mathbb{Z}.
	$$
	Evenness of $f_{\nu,R}(z)$ with respect to $z$ follows from the relations \cite[p.~75]{watson}
	$$
	J_\nu(-z)=e^{i\pi \nu}J_\nu(z)
	\quad \text{and} \quad
	Y_\nu(-z)=e^{-i\pi \nu}Y_\nu(z) + 2i \cos (\nu \pi) J_\nu(z),
	\quad z \in \mathbb{C}.
	\pushQED{\qed}
	\qedhere
	$$
\end{proof}

\begin{lemma}
	Any zero $a_{\nu,k}$ of $f_{\nu,R}$ satisfies \eqref{eq:willis}.
\end{lemma}
\begin{proof}
	Let us take any zero $z := a_{\nu,k}$.
	If neither $J_{\nu}(z)=J_\nu(Rz)=0$ nor $Y_{\nu}(z)=Y_\nu(Rz)=0$, then \eqref{eq:willis} is proved in \cite{willis}. 
	Assume that $J_{\nu}(z)=J_\nu(Rz)=0$. 
	Note that zeros of $J_\nu$ and $Y_\nu$ are interlacing, which implies that $Y_{\nu}(z), Y_\nu(Rz) \neq 0$. 
	Therefore, we can rewrite $f_{\nu,R}(z)=0$ in the form
	$$
	\frac{J_\nu(z)}{Y_\nu(z)} - \frac{J_\nu(Rz)}{Y_\nu(Rz)} = 0.
	$$
	Recalling that any zero of $f_{\nu,R}$ is simple (see, e.g., \cite{cochran}), the rate of change of $z=z(\nu)$ with respect to $\nu$ can be found from the total derivative
	$$
	\left.\frac{\partial}{\partial x}\left(\frac{J_\nu(x)}{Y_\nu(x)} - \frac{J_\nu(Rx)}{Y_\nu(Rx)}\right)\right|_{x=z} \frac{dz}{d\nu} + 
	\left.\frac{\partial}{\partial \nu}\left(\frac{J_\nu(x)}{Y_\nu(x)} - \frac{J_\nu(Rx)}{Y_\nu(Rx)}\right)\right|_{x=z} = 0.
	$$
	Then, using the relations \cite[(1), p.~76]{watson} and \cite[(2), p.~444]{watson}, we formally derive
	$$
	\frac{dz}{d\nu} = \frac{2z}{\left(\frac{Y_\nu^2(z)}{Y_\nu^2(Rz)}-1\right)}
	\left(\frac{Y_\nu^2(z)}{Y_\nu^2(Rz)} \int_0^\infty K_0(2Rz \sinh t) e^{-2\nu t}\, dt - \int_0^\infty K_0(2z \sinh t) e^{-2\nu t}\, dt\right).
	$$
	Finally, recalling that $J_{\nu}(z)=J_\nu(Rz)=0$, $R>1$, and that $J_\nu^2(x) + Y_\nu^2(x)$ decreases with respect to $x>0$ \cite[p.~446]{watson}, we have
	$$
	\frac{Y_\nu^2(z)}{Y_\nu^2(Rz)} = \frac{J_\nu^2(z) + Y_\nu^2(z)}{J_\nu^2(Rz) + Y_\nu^2(Rz)} > 1,
	$$
	and hence \eqref{eq:willis} follows.
	The proof of \eqref{eq:willis} in the case $Y_{\nu}(z)=Y_\nu(Rz)=0$ can be handled in much the same way as above; see also \cite{willis}.
\end{proof}

\bigskip
\noindent
{\bf Acknowledgements.}
This research has been supported by the project 18-03253S of the Grant Agency of the Czech Republic, and by the project LO1506 of the Czech Ministry of Education, Youth and Sports.
The author thanks the anonymous colleague who pointed out the validity of equation \eqref{eq:analit_1}.
Moreover, the author is grateful to the anonymous referee for constructive remarks and valuable suggestions which helped to improve the manuscript.

\end{document}